\documentclass[]{aspm}

\articleinfo{}{}{}

\usepackage{verbatim}
\usepackage{amssymb}
\usepackage{amsbsy}
\usepackage{amscd}
\usepackage{amsmath}
\usepackage{amsthm}
\usepackage[mathscr]{eucal}

\newtheorem{theorem}{Theorem}[section]
\newtheorem{lemma}[theorem]{Lemma}
\newtheorem{corollary}[theorem]{Corollary}

\theoremstyle{definition}

\newtheorem{example}[theorem]{Example}

\theoremstyle{remark}
\newtheorem{remark}[theorem]{Remark}

\newcommand{\abs}[1]{\lvert#1\rvert}
\def\norm#1{\left\Vert#1\right\Vert}
\def\RUCB{{\mathrm{RUCB}\,}}

\def\e{{\varepsilon}}
\def\cl{{\mathrm{cl}\,}}

\def\SS{{\mathcal{S}}}
\def\E{{\mathbb E}}
\def\N{{\mathbb N}}
\def\R{{\mathbb R}}
\def\C{{\mathbb C}}
\def\Q{{\mathbb Q}}
\def\s{{\mathbb S}}

\def\I{{\mathbb I}}
\def\T{{\mathbb T}}
\def\Aut{{\mathrm{Aut}\,}}
\def\Lip{{\mathrm{Lip}\,}}
\def\Homeo{{\mathrm{Homeo}\,}}
\def\me{{\mathrm{me}}}

\title[Concentration and actions]{Concentration of measure and whirly actions of Polish groups}

\author[V. Pestov]{Vladimir Pestov}

\address{Department of Mathematics and Statistics, University of Ottawa, 585 King Edward Avenue, Ottawa, Ontario K1N 6N5, Canada}

\email{vpest283@uottawa.ca}

\rcvdate{}
\rvsdate{}



\begin{document}

\begin{abstract}
A weakly continuous near-action of a Polish group $G$ on a standard Lebesgue measure space $(X,\mu)$ is whirly if for every $A\subseteq X$ of strictly positive measure and every neighbourhood $V$ of identity in $G$ the set $VA$ has full measure. This is a strong version of ergodicity, and locally compact groups never admit whirly actions. On the contrary, every ergodic near-action by a Polish L\'evy group in the sense of Gromov and Milman, such as $U(\ell^2)$, is whirly (Glasner--Tsirelson--Weiss).
We give examples of closed subgroups of the group $\Aut(X,\mu)$ of measure preserving automorphisms of a standard Lebesgue measure space (with the weak topology) whose tautological action on $(X,\mu)$ is whirly, and which are not L\'evy groups, thus answering a question of Glasner and Weiss. 
\end{abstract}

\maketitle

\section{Introduction}
In recent decades methods of asymptotic geometric analysis (the concentration of measure phenomenon) have been put to use to study dynamical properties of Polish group actions.
The story begins with a paper of Gromov and Milman \cite{GrM} who
had noticed the following.

Let $G$ be a Polish topological group
which admits an approximating chain of compact subgroups $K_n$, $n=1,2,3,\ldots$, that is,
\[K_1< K_2< \ldots< K_n<\ldots,\]
and
\[\cl \bigcup_{i=1}^\infty K_n = G.\]
(Such are, among others, the unitary group $U(\ell^2)$ with the strong operator topology and the group $\Aut(X,\mu)$ of measure-preserving automorphisms of a standard Lebesgue measure space with the weak topology.)

Denote by $\mu_n$ the normalized Haar measures on the compact subgroups $K_n$, and let $d$ be a compatible right-invariant metric on $G$. Suppose the family of metric spaces with measure $(K_n,d\vert_{K_n},\mu_n)$ exhibits the phenomenon of concentration of measure (is a L\'evy family): for every sequence $A_n\subseteq K_n$ of Borel subsets whose measures are uniformly bounded away from zero,
\[\varliminf_{n\to\infty}\mu_n(A_n)>0,\]
and for every $\e>0$, one has
\[
\lim_{n\to\infty}\mu_n\left(A_n\right)_{\e}=1.\]
(Here $A_\e$ denotes the $\e$-neihbourhood of $A$.) Then $G$ is called a  {\em L\'evy group.} 

Let us say that a topological group $G$ has the {\em fixed point on compacta property} if every continuous action of $G$ on a compact space has a common fixed point.

\begin{theorem}[Gromov--Milman \cite{GrM}, Glasner \cite{Gl1}]
\label{th:grm}
Every L\'evy group has the fixed point on compacta property. 
\end{theorem}

(The result in \cite{GrM} was obtained under an additional restriction on the action, later removed by Glasner \cite{Gl1}.) 
We will give a simple proof, cf. Theorem \ref{th:fpc} below.
Notice that no locally compact group other than $\{e\}$ has the fixed point on compacta property (Granirer and Lau \cite{GraL}). 

The known examples of L\'evy groups include $U(\ell^2)$ \cite{GrM}, the groups $L^0(X,\mu;K)$ of measurable maps from a standard Lebesgue measure space to a compact group equipped with the topology of convergence in measure (Glasner \cite{Gl1}, Furstenberg and Weiss, unpublished), more generally the unitary groups of hyperfinite von Neumann algebras with the ultraweak topology \cite{GP1,GP2}, the group $\Aut(X,\mu)$ and various other groups of measure preserving and nonsingular transformations ({\em ibid.}), etc. Notice that the class of groups with the fixed point on compacta property is wider than the class of L\'evy groups \cite{P98a,KPT}. For an overview of the theory, see \cite{P06}.

The L\'evy property has consequences in the context of ergodic theory as well, as discovered by Glasner, Tsirelson and Weiss \cite{GTW}, see also \cite{GW}. Here are the relevant definitions.
The symbol $(X,\mu)$ in our paper will always denote a standard Lebesgue measure space, that is, a separable Borel space equipped with a non-atomic Borel probability measure. The group $\Aut(X,\mu)$ of measure-preserving automorphisms of $(X,\mu)$ will be always equipped with the {\em weak} ({\em coarse}) topology, that is, the weakest topology making every mapping of the form 
\[\Aut(X,\mu)\ni g\mapsto \mu(A\bigtriangleup gA)\in\R\]
continuous, where $A\subseteq X$ is a Borel subset. 
A (measure-preserving) {\em weakly continuous near-action} of a Polish group $G$ on a measure space $(X,\mu)$ is a continuous homomorphism from $G$ to the group $\Aut(X,\mu)$ equipped with the weak topology. Such a near-action is {\em whirly} if for every non-null measurable subset $A\subseteq X$ and each neighbourhood $V$ of identity in $G$, the set $VA$ has full measure.
(The set $VA$ is defined unambiguously, up to a null set, by selecting a countable dense subset $\{v_n\}_{n=1}^\infty\subseteq V$ and posing $VA = \cup_{n=1}^\infty v_nA$.) Equivalently, a near-action is whirly if for every two sets $A$ and $B$ of strictly positive measure there is a $g\in G$ arbitrarily close to identity and such that $gA\cap B$ has a strictly positive measure.

Clearly, every whirly near-action is ergodic, but the converse is not true: for instance, no locally compact group admits non-trivial whirly near-actions. This is essentially a reformulation of a theorem due to Mackey, Vara\-darajan, and Ramsay (cf. references in \cite{GTW}) stating that every near-action of a locally compact group has a Borel (point, spatial) realization. On the contrary, for infinite-dimensional groups whirly near-actions are a common occurence.

\begin{theorem}[Glasner, Tsirelson and Weiss \cite{GTW}]
Every ergodic measure-preserving weakly continuous near-action of a Polish L\'evy group on a standard Lebesgue space is whirly. 
\label{th:gtwwhirly}
\end{theorem}

This result can be equivalently stated as follows.

\begin{theorem}[Glasner, Tsirelson and Weiss \cite{GTW}]
Let $G$ be a Polish L\'evy group and $X$ a compact $G$-space. Then every invariant probability measure on $X$ is supported on a set of fixed points.
\label{th:gtwfixed}
\end{theorem}

It can be shown that an amenable Polish group satisfying the conclusion of either theorem has the fixed point on compacta property. The converse is not true, as noted in \cite{GTW}. For instance, the group $\Aut(\Q,\leq)$ of order-preserving bijections of the rationals equipped with the topology of simple convergence on $\Q$ with discrete topology has the fixed point on compacta property \cite{P98a}. As noted in \cite{GTW}, this group acts in a measure-preserving way on the Cantor space $\{0,1\}^{\Q}$ equipped with the product measure, which is of course ergodic. 

The authors of \cite{GW} construct examples of non-L\'evy second countable groups $G$ which admit whirly actions on a standard Lebesgue measure space, among them the additive group of $\ell^2$. However, in those examples either the action homomorphism $G\to\Aut(X,\mu)$, while continuous, is not a topological embedding, or else the group $G$ is not Polish. This had motivated the following question \cite{GW}: suppose $G$ is a closed subgroup of $\Aut(X,\mu)$ such that the tautological action of $G$ on $(X,\mu)$ is whirly. Is $G$ necessarily a L\'evy group?

The main aim of this paper is to answer the question in the negative. We show that Theorems \ref{th:gtwwhirly} and \ref{th:gtwfixed} can be extended to a wider class of Polish groups which, while not necessarily L\'evy, retain a sufficient amount of concentration, witnessed by nets of probability measures that at the same time concentrate and converge to invariance. Such groups need not contain any compact subgroups. 
Among them, are the groups $G=L^0(X,\mu;H)$ of measurable maps from a Lebesgue space to an arbitrary amenable second-countable locally compact group $H$. Those groups embed into $\Aut(X,\mu)$ as closed topological subgroups in such a way that the resulting action of $G$ on $(X,\mu)$ is whirly, thus answering the question by Glasner and Weiss. 

\section{Measures on compact $G$-spaces}

Let $G$ be a topological group. 
The {\em right uniform structure} on $G$ has a basis of entourages of the diagonal of the form
\[V_R=\{(x,y)\in G\times G\colon xy^{-1}\in V\},\]
where $V$ runs over a neighbourhood basis at the identity of $G$. For every element $x$ in a compact $G$-space $X$, the orbit map 
\[G\ni g\mapsto gx\in X\]
is right uniformly continuous. 

By $\RUCB(G)$ we will denote the commutative $C^\ast$-algebra of all bounded complex valued right uniformly continuous functions on $G$.

\begin{remark}
\label{r:denses}
In this article, all acting groups $G$ are assumed to be Polish, that is, separable completely metrizable topological groups. Some results of topological dynamics and ergodic theory, usually stated and proved for countable discrete acting groups, can be extended readily to Polish groups by choosing a dense countable subgroup $\tilde G$ of $G$ and then using continuity of the action.
\end{remark}

A compact $G$-space is {\em topologically transitive} if for every pair of non-empty open subsets $U,V\subseteq X$ there is $g\in G$ with $gU\cap V\neq\emptyset$, and {\em point transitive} if there exists a point $x_0\in X$ (said to be a {\em transitive point}) whose orbit is dense in $X$. For metrizable compact $G$-spaces $X$ the two notions coincide. (Cf. \cite{Gl2}, 1.1, p. 14, where it is stated for countable discrete acting groups, but see our remark \ref{r:denses}.) A point transitive space $X$ together with a distinguished transitive point $x_0$ is called a $G${\em -ambit}. 

The maximal ideal space of $\RUCB(G)$ is denoted $\SS(G)$ and called the {\em greatest ambit} of $G$. It contains a topological copy of $G$ as a dense subspace, is equipped with a canonical action of $G$ extending the action of $G$ on itself by left translations, and contains a distinguished element whose orbit is dense (the identity of $G$). Every function $f\in\RUCB(G)$ extends in a unique way to a continuous function on the greatest ambit, and conversely, the restriction of every $f\in C(\SS(G))$ to $G$ is bounded right uniformly continuous. The greatest ambit is non-metrizable unless $G$ is a compact metrizable group, in which case $\SS(G)=G$.

When considering probability measures on compact spaces which are not necessarily metrizalbe, we assume that they are regular Borel. Essentially, we are always talking of Radon probability measures, that is, normalized states on $C(X)$, which by the Riesz representation theorem correspond to regular Borel probability measures on $X$.

We will use the symbol $\Rightarrow$ to denote the weak convergence of measures. For probability measures, weak convergence only depends on the topology of the underlying space (and not on the uniformity), cf. \cite{parthasarathy}, Thm. II.6.1, and so there can be no ambiguity. When considering weak convergence 
of Radon (signed) measures in Sect. \ref{s:invariance}, we take as the test functions uniformly continuous bounded functions on a metric space, and {\em right} uniformly continuous bounded functions on a Polish group. 

If $G$ is a metrizable group, equipped with a compatible right-invariant metric $d$, then the Lipschitz functions are uniformly dense in $\RUCB(G)$ (cf. e.g. \cite{miculescu}, theorem 1, modulo the well-known fact that every $1$-Lipschitz real function defined on a subset of a metric space $X$ extends to a $1$-Lipschitz function on $X$). 
Therefore, $1$-Lipschitz functions, or rather their unique extensions over $\SS(G)$, can serve as test functions for weak convergence of measures on the greatest ambit of $G$. 

\begin{lemma}
Let $G$ be a topological group equipped with a right invariant compatible metric $d$. Suppose $f\in C(\SS(G))$ and $f\vert_G$ is $1$-Lipschitz. Then for every $x\in \SS(G)$ the real function $\,_xf$ on $G$, defined by
\[G\ni g\mapsto \,_xf(g)=f(gx)\in \R,\]
is $1$-Lipschitz as well.
\label{l:lipschitz}
\end{lemma}

\begin{proof}
Let $x=\lim_{\beta}x_{\beta}$, $x_{\beta}\in G$. Since the action of $G$ on $X$ is continuous, we have for all $g,h\in G$
\begin{eqnarray*}
\abs{\,_xf(g)-\,_xf(h)}&=&\abs{f(gx)-f(hx)}\\
&=& \abs{f(g\lim_{\beta}x_\beta) - f(h\lim_{\beta}x_{\beta})} \\
&=&\lim_{\beta}\abs{f(gx_{\beta}) - f(hx_{\beta})}\\
&\leq& \varlimsup_{\beta} \,d(gx_{\beta},hx_{\beta})\\ 
&=& d(g,h).\end{eqnarray*}
\end{proof}

Let $G$ be a Polish group and let $X$ be a compact $G$-space with the corresponding action
\[\tau\colon G\times X\to X.\]
Let now $\nu$ be a Borel probability measure on $G$, and $\mu$ a regular Borel probability measure on $X$. We define their convolution as the push-forward of the product measure $\nu\times\mu$ along the action $\tau$,
\[\nu\ast\mu = \tau_{\ast}(\nu\times\mu)\equiv (\nu\times\mu)\tau^{-1},\]
or, more explicitely, by the condition
\[\int f\,d(\nu\ast\mu) = \int f(gx)\,d(\nu(g)\times\mu(x))\]
for every $f\in C(X)$. Thus, $\nu\ast\mu$ is a Radon probability measure on $X$.

We are principally interested in the case where $X=(X,x_0)$ is a $G$-ambit, in particular $X=\SS(G)$ is the greatest ambit. In the latter case, if $\mu$ is supported on $G$, the convolution $\nu\ast\mu$ is the usual convolution of two probability measures on a metric group $G$, cf. \cite{parthasarathy}, Ch. III.

\section{On the concept of a whirly group}

The following observation is largely contained in \cite{GTW,GW}.

\begin{theorem}
\label{th:equivalent}
For a Polish group $G$, consider the following assertions:
\begin{enumerate}
\item Every ergodic weakly continuous near-action of $G$ on a standard Lebesgue measure space is whirly.
\item Every invariant Borel probability measure on every metrizable compact $G$-space is supported on the set of fixed points. 
\item Same, for every point-transitive metrizable $G$-space $X$.
\item Every invariant regular Borel probability measure on the greatest ambit $\SS(G)$ is supported on the set of fixed points. 
\end{enumerate}
Then
\[(1)\iff (2)\iff(3)\Rightarrow (4),\]
and if in addition $G$ is amenable, then
\[(1)\iff(2)\iff(3)\iff(4).\]
\end{theorem}

\begin{proof}
(1)$\Rightarrow$(2): By contraposition. Let $X$ be a compact metrizable $G$-space $X$ and $\mu$ an invariant measure on $X$ which is not supported on a set of fixed points. An ergodic decomposition of $\mu$ performed using a dense countable subgroup of $G$ as in \cite{P06}, proof of Thm. 7.1.5 (pp.153-154), allows us to assume that $\mu$ is an ergodic probability measure which is not supported on a single fixed point. According to Proposition 3.3 in \cite{GTW}, the action of $G$ on the measure space $(X,\mu)$ is not whirly, because it admits a spatial realization. And of course one can assume without loss in generality that $(X,\mu)$ is non-atomic (by multiplying it with the interval upon which $G$ acts trivially, if necessary).
\par
(2)$\Rightarrow$(3): trivially true.
\par
(3)$\Rightarrow$(1): Again, a proof by contraposition. According to Thm. 3.10 in \cite{GW}, every measure preserving near-action that is not whirly admits a non-trivial spatial factor, $Y$, which can be assumed to be a compact metrizable $G$-space (Becker and Kechris, \cite{BK}, Thm. 5.2.1, cf. Thm. 0.4 in \cite{GTW}). If the original measure $\mu$ was $G$-ergodic, so will be its direct image, $\nu$, on $Y$. Assume that the measure $\nu$ is full, by replacing $Y$, if necessary, with the support of $\nu$. By applying Thm. 4.27 in \cite{Gl2} to a countable dense subgroup $\tilde G$ of $G$, we conclude that $Y$ is a topologically transitive $G$-space, and so point transitive.
\par
(2)$\Rightarrow$(4): Suffices to prove that, given a $f\in C(\SS(G))$ and $g\in G$, one has
\[\int \abs{f-\,^gf}\,d\mu=0.\]
Let $A_f$ denote the smallest invariant unital $C^\ast$-subalgebra of $\RUCB(G)$ containing $f$. This algebra is separable, and the space of maximal ideals $X_f$ is therefore a metrizable compact $G$-space. 
Furthermore, the function $f$ factors through the canonical equivariant continuous map $\pi_f\colon\SS(G)\to X_f$, that is, for some continuous $\bar f\colon X_f\to\C$, we have $f=\bar f\circ \pi_f$. Denote $\mu_f=\mu\pi_f^{-1}$. This is an invariant probability measure on $X_f$.
Using our assumptions, we connclude:
\[\int \abs{f-\,^gf}\,d\mu=\int \abs{\bar f-\,^g\bar f}\,d\mu_f=0.\]
\par
(4)$\Rightarrow$(3): Assume $G$ is amenable. Let $(X,x_0)$ be a metrizable $G$-ambit, and let $\mu$ be an invariant probability measure on $X$. The inverse image of $\mu$ under the (surjective) morphism of ambits $\SS(G)\to X$, $e\mapsto x_0$ is a weak$^\ast$-compact family of probability measures which is non-empty (by the Hahn-Banach theorem), $G$-invariant, and convex. The action of $G$ on the space of probability measures is affine, meaning there is an invariant measure $\nu$ on the greatest ambit, whose direct image is $\mu$. Since $\nu$ is supported on a set of fixed points, so is its direct image. 
\end{proof}

\begin{example}
The implication (4)$\Rightarrow$(3) in general fails for non-amenable Polish groups, e.g. the free group $F_2$ with the discrete topology satisfies (4) on a logical technicality, but of course admits plenty of ergodic actions and no nontrivial whirly ones.
\end{example}

\begin{example}
A Polish group $G$ satisfying the equivalent conditions $(1)-(3)$ need not have the fixed point on compacta property, nor be amenable. 

As observed in \cite{GTW}, Rem. 1.7, if a Polish group admits no non-trivial weakly continuous near-actions on the standard Lebesgue space --- equivalently, no non-trivial strongly continuous unitary representations --- then $G$ satisfies the condition (1) in Thm. \ref{th:equivalent} in a trivial way. 

According to Megrelishvili \cite{megrelishvili01}, the Polish group $\Homeo_+[0,1]$ of homeomorphisms of the unit interval fixing the endpoints, with the compact-open topology, does not even admit non-trivial strongly continuous isometric representations in reflexive Banach spaces. Now, this group is amenable, as it has the fixed point on compacta property \cite{P98a}.

However, a sister group $\Homeo_+(\s^1)$ of orientation-preserving homeomorphisms of the circle is not amenable (as there is obviously no invariant probability measure on the circle) and also admits no nontrivial unitary representations, as a direct consequence of Megrelishvili's theorem: this group equals the closed normal subgroup generated by the stabilizer of any point, which stabilizer is isomorphic to $\Homeo_+[0,1]$. 
\end{example}

However, the following is immediate.

\begin{corollary}
Let a Polish group $G$ satisfy any of the equivalent conditions (1)-(3) in Theorem \ref{th:equivalent}. If $G$ is amenable, then it has the fixed point on compacta property. \qed
\end{corollary}

Let us say that a Polish group is {\em whirly amenable} if $G$ is amenable and satisfies equivalent conditions (1)-(3) (and thus (4)) in Thm. \ref{th:equivalent}.
Inside of the class of Polish groups, we have

\[\mbox{L\'evy groups}\subsetneqq\mbox{whirly amenable groups}\subsetneqq\mbox{groups with  f.p.c. property.}\]

The fact that the inclusion on the l.h.s. is strict, will only be established at the end of this article. 

\section{L\'evy nets of measures}

Let $(\nu_{\beta})$ be a net of regular Borel probability measures on a uniform space $(X,{\mathcal U})$. We say that the net $(\nu_{\beta})$ {\em concentrates,} or forms a {\em L\'evy net}, if for every net of Borel sets $(A_{\beta})$, $A_{\beta}\subseteq G$, with the property \[\varliminf\mu_{\beta}(A_{\beta})>0\]
and every entourage of the diagonal $V\in{\mathcal U}_X$, one has \[\lim\nu_{\beta}\left(V[A]\right)=1.\]

We need the following quantitative measure of concentration, introduced in \cite{GrM}. Let $(X,d,\mu)$ be a metric space with measure. The {\em concentration function} of $X$ is defined for all $\e\geq 0$ by the rule

\[\alpha_X(\e)=
\begin{cases} \frac 12, & \mbox{ if $\e=0$,} \\
1-\inf\left\{\mu\left(B_\e\right) \colon
B\subseteq X, ~~ \mu(B)\geq\frac 12\right\}, &
\mbox{if $\e>0$.}
\end{cases}
\]
A net of measures $(\nu_\beta)$ on a metric space $(X,d)$ is a L\'evy net if and only if 
\[\alpha_{(X,d,\nu_{\beta})}(\e)\to 0\mbox{ pointwise in $\e$ on $(0,+\infty)$}.\]

Remark that if $f\colon X\to Y$ is a uniformly continuous map and $(\nu_{\beta})$ a 
L\'evy net of probability measures on $X$, then the net $(f_{\ast}\mu_\beta)$ of push-forward measures is a 
L\'evy net in $Y$. In particular, if a net of measures 
concentrates in a Polish group $G$, then it will 
concentrate in the greatest ambit $\SS(G)$ as well.

\begin{example}
Let $G$ be a L\'evy group in the sense of Gromov and Milman, with a corresponding sequence of approximating compact subgroups $(K_n)$. The sequence of the probability Haar measures, $\mu_n$, on subgroups $K_n$ is a L\'evy net, simply as a part of the definition of a L\'evy group.
\end{example}

\begin{remark}
As shown by Farah and Solecki \cite{FS}, a L\'evy group can admit an approximating sequence of compact subgroups $(H_n)$ whose probability Haar measures $\nu_n$ {\em do not concentrate}, that is, do not form a L\'evy sequence. Still, such sequences of subgroups retain a weaker form of concentration. In particular, as proved by the same authors, for every big set $A\subseteq G$ (that is, one has $KA=G$ for a suitable compact $K\subseteq G$) and every neighbourhood $V\ni e$, one has $\lim_{n\to\infty}\nu_n(VA)=1$. It remains unknown if every Polish group admitting an approximating sequence of compact subgroups with this property is a L\'evy group ({\em ibid.})
\end{remark}

\begin{example}
\label{ex:kn}
For a Polish group $H$, denote $L^0(X,\mu;H)$ the collection of all (equivalence classes of) measurable maps from a standard Lebesgue measure space $(X,\mu)$ to $H$, equipped with the pointwise algebraic operations and the topology of convergence in measure. This topology is given by the following right-invariant metric (we follow Gromov's notation \cite{Gr}, p. 115):
\[\me_1(f,g) =\inf\left\{\e>0\colon \mu\{x\in X\colon d(f(x),g(x))>\e\}<\e\right\},\]
where $d$ is an arbitrary compatible right-invariant metric on $H$.

The group $L^0(X,\mu;H)$, first considered in \cite{HM}, is a Polish group, and moreover, as a topological space, is homeomorphic to $\ell^2$ provided $H$ is nontrivial \cite{BP}. 

Note that if the group $H$ is compact, then $L^0(X,\mu;H)$ is a L\'evy group (Glasner \cite{Gl1}; Furstenberg and Weiss, unpublished). 

Assume, more generally, that $H$ is an amenable locally compact second countable group.
The following L\'evy net of probability measures on $L^0(X,\mu;H)$, studied by the present author in \cite{P05}, subs. 2.5, will be used in this paper.

It is convenient to identify $(X,\mu)$ with the interval $\I=[0,1]$ equipped with the Lebesgue measure. The indexing set for our net consists of all triples $(F,\delta,n)$, where $F\subseteq H$ is a compact subset, $\delta>0$, and $n\in\N_+$, 
directed in a natural way ($\subseteq,\geq,\leq$). 
Fix a left-invariant sigma-finite Borel measure $\nu$ on $G$. To every pair $(F,\delta)$ as above assign a compact F\o lner set $K=K_{F,\delta}\subseteq G$ so that
\[
\forall g\in F,~~\frac{\nu(gK\bigtriangleup K)}{\nu(K)}<\delta.
\]
The $n$-th topological power $K^n$ can be identified with the set of $K$-valued functions constant on every interval $[i/n, (i+1)/n)$, $i=0,1,\dots,n-1$.
Denote by $\nu_{F,\delta,n}$ the product measure $(\nu\vert_K)^n$ normalized to one.

To see that the net of measures $(\nu_{F,\delta,n})$ is L\'evy, observe the following. The family of measure spaces $(X^n,(\nu\vert_K)^n)$, equipped with the normalized Hamming metric
\[d_n(x,y)=\frac 1n \left\vert \{i=1,\ldots,n\colon x_i\neq y_i\}\right\vert,\]
is a well-known example of a L\'evy family, satisfying
\[\alpha_{K^n}(\e)\leq e^{-\e^2n/8}.\]
(Cf. e.g. \cite{ledoux}, Eq. (1.24) on p. 18, where the formula has to be adjusted because the distance is non-normalized).

It is easy to see that, for all $f,g\in K^n$,
\[\me_1(f,g)\leq d_n(f,g),\]
and consequently one has for every compact $F\subseteq H$, $0<\delta<1$, and $n\in\N$:
\[\alpha_{(L^0(X,\mu;H),\me_1,\nu_{F,\delta,n})}(\e)\leq 2e^{-\e^2n/2}.\]
\end{example}

\begin{remark}
Of course, strictly speaking, every topological group admits a L\'evy sequence of regular probability measures, indeed any sequence of point masses is such. What makes the above examples interesting, is that the corresponding L\'evy nets of measures converge to invariance in a suitable sense, to be discussed shortly. 
\end{remark}

For a measurable function $f\colon X\to\R$ denote by $M=M_f$ the median value of $f$, that is, a (possibly non-unique) real number satisfying
\[\mu\{x\in X\colon f(x)\leq M\}\geq\frac 12\mbox{ and }
\mu\{x\in X\colon f(x)\geq M\}\geq\frac 12.\]
The following result from \cite{GrM} is by now standard. 

\begin{lemma}
Let $f$ be a $1$-Lipschitz function on a metric space with measure $X$. Then for every $\e>0$
\[\mu\{x\in X\colon \abs{f(x)-M_f}>\e\}\leq 2\alpha(\e).\]
\end{lemma}

As an immediate consequence, we get the following.

\begin{lemma}
\label{l:immediate}
Let $X=(X,d,\mu)$ be a metric space with measure of diameter $\leq 1$, and let $f$ be a $1$-Lipschitz real-valued function on $X$. Then for every $\e>0$
\[\left\vert\E_{\mu}(f)-M_f\right\vert < \e + 2\alpha(\e).\]
\qed
\end{lemma}

Since for L\'evy families $\alpha(\e)\to 0$ for each $\e>0$, we obtain more or less directly:

\begin{lemma}
Let $(X_\beta,d_\beta,\mu_\beta)$ be a L\'evy net of metric spaces with measure. For every $\beta$ let $B_\beta\subseteq X_\beta$ be a Borel subset, so that
\[\varliminf_{\beta}\, \mu_\beta(B_\beta)>0.\]
Then
\[\sup_{f\in\Lip_1(X)}\left\vert \frac{1}{\mu_{\beta}(B_{\beta})}\int_{B_\beta}f\,d\mu_\beta - \int_{X_\beta}f\,d\mu_\beta\right\vert \to 0,\]
where $\Lip_1(X)$ is the family of all $1$-Lipschitz real-valued functions on $X$. \qed
\label{l:integration}
\end{lemma}

Next we will show that every L\'evy net on a Polish group serves as an approximate left identity for invariant measures on the greatest ambit. 

\begin{lemma}
Let $(\nu_\beta)$ be a L\'evy net of probability measures on a Polish group $G$, and let $\mu$ be an invariant regular Borel probability measure on the greatest $G$-ambit $\SS(G)$. Then
\[\nu_{\beta}\ast\mu \Rightarrow \mu.\]
\label{l:convolution}
\end{lemma}

\begin{proof}
We want to show that for every $f\in C(\SS(G))$,
\[\int f\,d(\nu_{\beta}\ast\mu) \to \int f\, d\mu.\]
It is enough to choose a right-invariant compatible metric $d$ on $G$, bounded by one, and verify the statement for the extensions by continuity of $1$-Lipschitz functions on $G$ over $\SS(G)$. Let $f\colon \SS(G)\to\R$ be such an extension, and  let $\e>0$ be arbitrary. Consider the function $f\circ\tau$ on $G\times X$:
\[f\tau(g,x) = f(gx).\]
For every $x\in X$, the function $\,_xf$ on $G$ obtained from $f\tau$ by fixing $x$ and letting $g$ vary, is $1$-Lipschitz by Lemma \ref{l:lipschitz}. Denote by $\alpha_{\beta}$ the concentration function of $(G,d,\nu_{\beta})$. Since $\alpha_\beta(\e)\to 0$, for sufficiently large $\beta$ and our fixed value of $\e$ one has
\[\alpha_\beta(\e)<\e,\]
and therefore (Lemma \ref{l:immediate})
\[\nu_{\beta}\left\{g\in G\colon \left\vert \,_xf(g) - \E_{\nu_{\beta}}\,(_xf)\right\vert >\e\right\}<3\e,\]
which holds uniformly in $x$. Consequently, by the Fubini theorem,
\[(\nu_{\beta}\times\mu)\left\{ (g,x)\in G\times \SS(G)\colon \left\vert \,_xf(g) - \E_{\nu_{\beta}}\,_xf \right\vert >\e\right\}<3\e,\]
and there exists $g_0\in G$ such that
\[\mu\left\{x\in \SS(G)\colon \left\vert f(g_0x) - \E_{\nu_{\beta}}\,(_xf) \right\vert >\e\right\}<3\e.\]
Since the measure $\mu$ is $G$-invariant, for all sufficiently large $\beta$
\begin{eqnarray*}
\left\vert \int f\,d(\nu_{\beta}\ast\mu)-\int f(x)\,d\mu\right\vert &=&
\left\vert \int f(gx)\,d\nu_{\beta}(g)d\mu(x)-\int f(g_0x)\,d\mu\right\vert \\
&=&\left\vert \int d\mu \left[\int d\nu_{\beta}\, \left(f(gx) - f(g_0x)\right)\right]\right\vert \\
&\leq& 4\e.
\end{eqnarray*}
\end{proof}

\section{\label{s:invariance}Convergence to invariance}

A net $(\nu_\beta)$ of probability measures on a topological group $G$ {\em converges to invariance in a weak sense} if for every $g\in G$ one has
\[g\ast\nu_\beta - \nu_\beta\Rightarrow 0.\]
(As we noted elsewhere, the convergence of signed measures is to be understood as weak convergence on the greatest ambit of $G$, that is, the weak$^\ast$ convergence with regard to bounded right uniformly continuous functions on $G$.)

It is easy to see that a Polish group $G$ is amenable if and only if it admits a net of probability measures weakly converging to invariance. The following result generalizes Theorem \ref{th:grm}.

\begin{theorem}
Let a Polish group $G$ admit a L\'evy net $(\nu_\beta)$ of probability measures weakly converging to invariance. Then $G$ has the fixed point on compacta property.
\label{th:fpc}
\end{theorem}

\begin{proof}
By proceeding to a subnet if necessary, we can assume that $(\nu_\beta)\Rightarrow \nu$, where $\nu$ is a probability measure on $\SS(G)$. This measure inherits from the approximating net the property that for every $1$-Lipschitz function $f$ and each $\e>0$, 
\[\nu\{x\in\SS(G)\colon \abs{f(x)-\E_{\nu}(f)}>\e\}<\e,\]
in other words, $f$ is constant $\nu$-a.e. This is only possible if $\nu$ is supported on a single point. On the other hand, for every $g\in G$ one clearly has
\[\nu \Leftarrow g\ast\nu_\beta\Rightarrow g\ast\nu,\]
so $\nu$ is $G$-invariant. We conclude: the support of $\nu$ is a $G$-fixed point in the greatest ambit.
\end{proof}

For our purposes, we need a stronger version of invariance.

Let $(X,d)$ be a metric space. Recall that the {\em Monge-Kantorovich distance} between two probability measures, $\mu$ and $\nu$, on $X$ is given by the formula
\begin{equation} 
\label{kantor}
d_{MK}(\mu_1,\mu_2)
= \inf\left\{\int_{X\times X} d(x,y)\, d\nu~ \colon~ 
\pi_{i,\ast}(\nu) = \mu_i,~ i=1,2\right\},
\end{equation} 
where $\pi_i$ are coordinate projections and $\pi_{i,\ast}(\nu)$ are marginals (direct images of the measure under projections). We refer to \cite{villani} as a general source.

Here is an alternative definition. Assume that the values of $d$ are bounded by one, and denote $\Lip(X)$ the vector space of all Lipschitz real-valued functions on $X$. The minimum of the uniform norm and the value of the smallest Lipschitz constant of a function $f$ defines a norm on $\Lip(X)$. Every probability measure on $X$ is an element of the dual space $\Lip(X)^{\ast}$, and according to a classical result of Kantorovich and Rubinstein,
\[d_{MK}(\mu_1,\mu_2)=\norm{\mu_1-\mu_2},\]
with regard to the dual norm. (Cf. e.g. \cite{villani}, remark 6.5.)

Now let $G$ be a topological group equipped with a compatible right-invariant metric, $d$. 
A net of probability measures $(\nu_{\beta})$ on $G$ {\em converges to invariance in the sense of Monge-Kantorovich distance} if for every $g\in G$
\[d_{MK}(g\ast\nu_{\beta},\nu_{\beta})\to 0.\]

\begin{remark}
It is enough to verify the above condition for all $g$ in a dense subset of $G$, as follows from the observation that 
\[d_{MK}(g\ast\mu,h\ast\mu)\leq d(g,h).\]
Indeed, if $f$ is a $1$-Lipschitz function on $G$, then
\begin{eqnarray*}
\left\vert \int f\,d(g\ast\mu) - \int f\,d(h\ast\mu)\right\vert &=&
\left\vert \int f(gx)\,d\mu(x) - \int f(hx)\,d\mu(x)\right\vert \\
&\leq& \sup_{x}\, d(gx,hx) \\
&=& d(g,h),
\end{eqnarray*}
and the duality result by Kantorovich--Rubinstein settles the matter.
\label{r:dense}
\end{remark}

\begin{remark}
Clearly, convergence to invariance in the sense of Monge-Kantorovich distance implies weak convergence to invariance. 
\label{r:implies}
\end{remark}

\begin{example}
Let $G$ be a Polish group admitting an approximating chain of compact subgroups $(K_n)$. Then the sequence $(\nu_n)$ of probability Haar measures on the groups $K_n$ converges to invariance in the Monge-Kantorovich sense with regard to an arbitrary compatible right-invariant metric $d$ on $G$.

Indeed, in view of Remark \ref{r:dense}, one can assume without loss in generality that $g\in K_N$ for some $N$, in which case for all $n\geq N$
\[g\ast\mu_n=\mu_n.\]
\end{example}

\begin{lemma}
\label{l:bigformulas}
Let $(\nu_\beta)$ be a L\'evy net of probability measures on a Polish group $G$ with the property that for every $g$ from a dense subset of $G$ and each $\beta$, one can choose a Borel subset $B_\beta\subseteq G$ in such a way that
\[\nu_\beta\vert_{B_\beta}  = \left(g\ast\nu_\beta\right)\vert {B_\beta}\]
and 
\[\varlimsup_\beta\,\nu_\beta(B_{\beta})>0.\]
Then the net $(\nu_\beta)$ is converging to invariance in the sense of Monge-Kantorovich, with regard to any right invariant compatible metric $d$ on $G$.
\end{lemma}

\begin{proof}
For every $1$-Lipschitz function $f$ on $G$ one has
\begin{eqnarray*}
\left\vert \int f\,d \nu_\beta\right. &-& \left.\int f\,d (g\ast\nu_\beta)\right\vert \leq
\left\vert \int f\,d \nu_\beta - \frac{1}{\nu_{\beta}({B_\beta})}\int_{B_\beta} f\,d \nu_\beta \right\vert + \\
&&
\left\vert  \frac{1}{\nu_{\beta}({B_\beta})}\int_{B_\beta} f\,d \nu_\beta - \frac{1}{(g\ast\nu_{\beta})({B_\beta})}\int_{B_\beta} f\,d (g\ast\nu_\beta)
\right\vert + \\
&& 
\left\vert \frac{1}{(g\ast\nu_{\beta})({B_\beta})}\int_{B_\beta} f\,d (g\ast\nu_\beta) - \int f\,d (g\ast\nu_\beta)\right\vert,
\end{eqnarray*}
where the first and the last term converge to zero uniformly in $f$ by force of Lemma \ref{l:integration}, while the middle term vanishes because of our assumption. Now conclude by duality.
\end{proof}

\begin{example}
Let $H$ be an amenable second countable locally compact group.
The net of probability measures $\nu_{F,1/n!,n!}$ on the Polish group $L^0(X,\mu;H)$, which is a subnet of the net $\nu_{F,\e,n}$
defined in Example \ref{ex:kn}, converges to invariance in the sense of Monge-Kantorovich distance (formed with regard to $\me_1$). This follows from Lemma \ref{l:bigformulas} modulo the following observation that (essentially) appears in \cite{P05}. Let $g$ be a simple function with regard to a uniform finite partition of the interval $[0,1]$ into $m$ subintervals. Then for all compact $F\ni g$ and all $n\geq m$, the restrictions of the measures $\nu_{F,1/n!,n!}$ and $g\ast\nu_{F,1/n!,n!}$ coincide on the set $(K\cap gK)^{n!}$, whose $\nu_{F,1/n!,n!}$-measure will either approach or exceed $e^{-1}$ as $n\to\infty$. 
\label{ex:principal}
\end{example} 

\begin{theorem}
Let $G$ be a Polish group admitting a L\'evy net of probability measures converging to invariance with regard to the Monge-Kantorovich distance. Then every invariant measure on a compact $G$-space is supported on the set of $G$-fixed points. 
\end{theorem}

\begin{proof}
We want to show that for every $f\in C(\SS(G))$ and all $g\in G$, the function $f$ and its $g$-translate $\,^gf$ are equal $\mu$-a.e., that is, 
\[\int \left\vert f-\,^gf\right\vert \,d\mu =0.\]
Assume, without loss of generality, that $f\vert_G$ is $1$-Lipschitz with regard to a right invariant compatible metric $d$ bounded by one. Let $\e>0$ be arbitrary. 
Due to the Monge-Kantorovich convergence of $(\nu_{\beta})$ to invariance, 
for some $\beta^\prime$ and all $\beta\geq\beta^\prime$ we have 
\begin{equation}
\label{eq:difference}
\left\vert \E_{\nu_{\beta}}(k) - \E_{\nu_{\beta}}(^gk)\right\vert <\e,\end{equation}
uniformly over all $1$-Lipschitz functions $k$, in particular, over all functions of the form $\,_xf$, $x\in\SS(G)$ (cf. Lemma \ref{l:lipschitz}).

Because of concentration, for sufficiently large $\beta$ we further have, uniformly in $x\in X$,
\[\nu_{\beta}\{h\colon \abs{\E_{\nu_\beta}(\,_xf)-\,_xf(h)}>\e \}<\e \mbox{ and }
\nu_{\beta}\{h\colon \abs{\E_{\nu_\beta}(^g_xf)-\,_xf(gh)}>\e \}<\e.
\]
Combining this with Eq. (\ref{eq:difference}), we conclude: given $x\in \SS(G)$, for all $h$ except possibly a set of $\nu_\beta$-measure $<3\e$, the value $\abs{\,_xf(h)-\,_xf(gh)}$ does not exceed $3\e$.
In view of Lemma \ref{l:convolution},
\begin{eqnarray*}
\int \left\vert f -\,^gf\right\vert\,d\mu &=& 
\lim_{\beta} \int \left\vert f(hx) -\,^gf(hx)\right\vert\,d\nu_\beta(h)\,d\mu(x) \\
&<& 6\e,
\end{eqnarray*}
and as $\e$ was arbitrary, the proof is finished.
\end{proof}

Taking into account Theorems \ref{th:equivalent} and \ref{th:fpc}, we obtain the following results, which serve as a source of ``whirly amenable'' Polish groups.

\begin{corollary}
Let a Polish group $G$ admit a L\'evy net $(\nu_\beta)$ of probability measures converging to invariance in the Monge-Kantorovich sense. Then every ergodic invariant regular Borel probability measure $\mu$ on an arbitrary compact $G$-space is a point mass. \qed
\end{corollary}

\begin{corollary}
Let $G$ be a Polish group admitting a L\'evy net of probability measures converging to invariance with regard to the Monge-Kantorovich distance. Then every ergodic weakly continuous near-action of $G$ on a standard Lebesgue probability space is whirly. \qed
\label{c:whirly}
\end{corollary}

\section{Gaussian near-actions}

Here is our main application.

\begin{theorem}
Let $H$ be an amenable second countable locally compact group, and let $(X,\mu)$ denote the standard Lebesgue measure space. 
Then the Polish group of maps $L^0(X,\mu;H)$ embeds into $\Aut(X,\mu)$ as a closed topological subgroup in such a way that the resulting action of $L^0(X,\mu;H)$ on $(X,\mu)$ is whirly. 
\label{th:main}
\end{theorem}

In the case $H=\T$, the result belongs to Glasner, Tsirelson and Weiss \cite{GTW}. 

\begin{example}
Let $H$ be a locally compact group without non-trivial compact subgroups. Then every monothetic subgroup of $H$ is infinite discrete by a well-known result. This property is clearly shared by $L^0(X,\mu;H)$, which group therefore also contains no non-trivial compact subgroups and is not a L\'evy group.  

Together with theorem \ref{th:main}, this observation answers in the negative a question by Glasner and Weiss, cf. problem 2 in \cite{GW}. For instance, the additive group of the topological vector space $L^0(X,\mu)$ embeds into $\Aut(X,\mu)$ as a closed topological subgroup in such a way that the tautological action on $(X,\mu)$ is whirly. 
\end{example}

In order to prove theorem \ref{th:main}, it is enough to embed $L^0(X,\mu;H)$ as a topological subgroup into $\Aut(X,\mu)$ in such a way that the resulting action on the measure space is ergodic. (A Polish topological subgroup is automatically closed in the ambient group.) This is achieved using the technique of Gaussian near-actions associated to unitary representations. Here is a brief summary of the relevant small part of the theory, for details we refer to \cite{Gl2}, sections 3.11-3.12 and \cite{kechris}, Appendix E, as well as numerous further references contained in those sources.

Let $(\R^{\infty},\gamma^{\infty})$ denote a countable infinite power of the real line equipped with the product of standard Gaussian measures with mean zero and variance one. This is a standard Lebesgue measure space.
The coordinate projections $p_i$, $i=1,2,\ldots$, form an orthonormal system in the real Hilbert space $L^2(\R^{\infty},\gamma^{\infty})$ and so span an isometrically isomorphic copy of $\ell^2_{\R}$, denoted ${\mathcal H}^{\colon 1\colon}$ and known as the {\em first Wiener chaos}. To every orthogonal operator $T$ on ${\mathcal H}^{\colon 1\colon}$ one can associate a unique measure-preserving automorphism $\tilde T$ of $(\R^{\infty},\gamma^{\infty})$ in such a way that the unitary operator corresponding to $\tilde T$ leaves ${\mathcal H}^{\colon 1\colon}$ invariant and coincides with $T$ on the latter space. The resulting monomorphism $O({\mathcal H}^{\colon 1\colon})\hookrightarrow \Aut(\R^{\infty},\gamma^{\infty})$ is an embedding of topological groups. 

If now $\pi$ is a strongly continuous orthogonal representation of a topological group $G$ in $\ell^2_{\R}$ (whose isomorphism with ${\mathcal H}^{\colon 1\colon}$ needs to be chosen), then $\pi$ gives rise to a weakly continuous near-action of $G$ on a standard Lebesgue measure space.
By considering a complexification of the first Wiener chaos, one constructs a similar extension for unitary representations instead of orthogonal ones (which is somewhat more convenient). The resulting near-action on $(\R^{\infty},\gamma^{\infty})$ is the {\em Gaussian near-action} corresponding to $\pi$.

\begin{theorem}[\cite{kechris}, p. 214; \cite{Gl2}, Theorem 3.59]
If a unitary representation of a countable group $\Gamma$ has no non-zero finite dimensional subrepresentations, then the corresponding Gaussian near-action of $G$ on $(\R^{\infty},\gamma^{\infty})$ is weakly mixing.
\label{th:mixing}
\end{theorem}

\begin{remark}
By considering the closure of $\Gamma$ inside of the orthogonal group, we conclude that the same remains true of Polish groups and their strongly continuous unitary representations.
\end{remark}

Now it remains only to notice the following.

\begin{lemma} 
Let $H$ be a second countable locally compact group. Then the Polish group $L^0(X,\mu;H)$ admits a topologically faithful unitary representation, that is, can be embedded into $U(\ell^2)$ as a topological subgroup.
\label{l:faithful}
\end{lemma}

\begin{proof}
Every locally compact group admits a topologically faithful unitary representation (in fact, one can use the left regular representation, cf. Proposition 2 in \cite{PT}). Choose such a representation, $\rho$, of $H$ in a separable Hilbert space ${\mathcal H}$. Now define a representation $\pi$ of $L^0(X,\mu;H)$ in $L^2(X,\mu;{\mathcal H})$ as follows: for every $f\in L^0(X,\mu;H)$ and $\psi\in L^2(X,\mu;{\mathcal H})$, set
\[\pi_f(\psi)(x) =\rho_{f(x)}(\psi(x)),\]
where the equality is understood $\mu$-a.e. Every operator $\pi_f(\psi)$ is clearly unitary. Fixing a simple function $\psi\in L^2(X,\mu;{\mathcal H})$, one can see that the corresponding orbit map 
\[L^0(X,\mu;H)\ni f\mapsto \pi_f(\psi)\in L^2(X,\mu;{\mathcal H})\]
is continuous at identity, and consequently $\pi$
is a strongly continuous representation. 

Fix a right-invariant compatible pseudometric $d$ on $H$, and let $\e>0$ be arbitrary. Since $\rho$ is topologically faithful, there are $\xi_1,\ldots,\xi_n\in {\mathcal H}$ and a $\delta>0$ such that
\[\forall g\in G,~~\mbox{ if }\norm{\rho_g(\xi_i)-\xi_i}<\delta\mbox{ for all }i,\mbox{ then }d(g,e)<\e.\]
Let $\bar\xi_i$ be a constant function on $X$ taking value $\xi_i$. For any function $f\in L^0(X,\mu;H)$, if
\[\forall i=1,2,\ldots,n~~\norm{\pi_f(\bar\xi_i)-\bar\xi_i}<\delta,\]
then one must have for every $i$
\[\norm{\rho_{f(x)}(\xi_i)-\xi_i}<\delta\]
on a set of measure at least $1-\e$, meaning that 
\[\me_1(f,e)<n\e.\]
This means that $\pi$ is an embedding of topological groups, 
\[\pi\colon L^0(X,\mu;H)\hookrightarrow U(L^2(X,\mu;{\mathcal H})).\]
\end{proof}

\begin{proof}[Proof of theorem \ref{th:main}]
Choose a topologically faithful unitary representation $\pi$ of $L^0(X,\mu;H)$ in a separable Hilbert space as in Lemma \ref{l:faithful}. Of course we can assume that $\pi$ contains no invariant vectors. Since $L^0(X,\mu;H)$ has the fixed point on compacta property (Th. \ref{th:fpc} and Rem. \ref{r:implies}), it is minimally almost periodic, and so $\pi$ contains no finite-dimensional subrepresentations. For this reason, the Gaussian near-action of $L^0(X,\mu;H)$ associated to $\pi$ is weakly mixing (Thm. \ref{th:mixing}) and in particular ergodic. By Corollary \ref{c:whirly} and Ex. \ref{ex:principal}, this near-action is whirly. At the same time, $L^0(X,\mu;H)$ sits inside of $\Aut(\R^\infty,\gamma^\infty)$ as a closed topological subgroup.
\end{proof}

\section{Some problems}

\begin{enumerate}
\item The second part of problem 2 in \cite{GW} remains open: let $T\in\Aut(X,\mu)$ be such that the near-action of the closed subgroup $\Lambda(T)$ generated by $T$ on $(X,\mu)$ is whirly, is $\Lambda(T)$ a L\'evy group?
\item Suppose $G$ is a closed subgroup of $\Aut(X,\mu)$ whose tautological action on $(X,\mu)$ is whirly. Is $G$ necessarily amenable?
\item Here is a broader interpretation of the original question of Glasner and Weiss: is it possible to give necessary and sufficient conditions for a Polish group $G$ to be whirly amenable in terms of concentration of measure, e.g. the existence of a L\'evy net of measures which converges in some sense either to invariance or to idempotency?
(Cf. the following result (\cite{parthasarathy}, Th. 3.1): idempotent probability measures $\mu$ on Polish groups ($\mu\ast\mu=\mu$) are exactly Haar measures on compact subgroups.)
\item Say that a strongly continuous unitary representation of a Polish group is whirly if the associated Gaussian near-action is whirly. Is it possible to characterize such representations? 
\end{enumerate}


\end{document}